\documentclass[12pt]{article}
\usepackage{amsmath, amsthm, amsfonts, amssymb}
\usepackage{mathrsfs}
\usepackage{bbm}
\usepackage{amssymb}
\usepackage{txfonts}
\usepackage{graphicx}
\usepackage{wrapfig}

\newlength{\defbaselineskip}
\newcounter{marnote}

\newcommand{\setlinespacing}[1]%
           {\setlength{\baselineskip}{#1 \defbaselineskip}}

\theoremstyle{plain}
\newtheorem{theorem}{Theorem}[section]

\newtheorem{prop}[theorem]{Proposition}
\theoremstyle{definition}
\newtheorem{definition}{Definition}[section]
\theoremstyle{remark}

\numberwithin{equation}{section}

\begin{document}

\title{Harmonic maps on domains with piecewise Lipschitz continuous metrics}

\author{Haigang Li\footnote{School of Mathematical Sciences, 
\& Laboratory of Mathematics and Complex Systems, Ministry of
Education, Beijing Normal University, Beijing 100875,
 P. R. China.}, \ \ Changyou Wang
\footnote{Department of Mathematics, University of Kentucky, Lexington, KY 40506, USA.}}
\date{}
\maketitle

\begin{abstract} For a bounded domain $\Omega$ equipped with a piecewise Lipschitz continuous
Riemannian metric $g$, we consider harmonic map from $(\Omega,g)$ to a compact Riemannian manifold $(N,h)\hookrightarrow\mathbb R^k$ without boundary.
We generalize the notion of stationary harmonic map and prove the partial regularity. We also discuss
the global Lipschitz and piecewise $C^{1,\alpha}$-regularity of harmonic maps from $(\Omega,g)$ to manifolds
that support convex distance functions.
\end{abstract}

\section{Introduction}

Throughout this paper, we assume that $\Omega$ is a bounded domain in $\mathbb{R}^{n}$, separated by  a $C^{1,1}$-hypersurface $\Gamma$ into two subdomains $\Omega^+$ and $\Omega^-$, namely, $\Omega=\Omega^+\cup\Omega^-\cup\Gamma$, and
$g$ is a piecewise Lipschitz metric on $\Omega$ that is  $g\in{C}^{0,1}(\Omega^{+})
\cap {C}^{0,1}(\Omega^{-})$ but discontinuous at any $x\in\Gamma$.
For example, $\Omega=B_{1}\subset\mathbb{R}^{n}$ is  the unit ball, $\Gamma=B_1\cap\{x=(x',0)\in\mathbb R^n\}$, and
$$\bar{g}(x)=
\begin{cases}
g_{0} &x\in{B}_{1}^{+}=\{x^{n}>0\}\cap{B}_{1},\\
kg_{0} &x\in{B}_{1}^{-}=\{x^{n}<0\}\cap{B}_{1},
\end{cases}
$$
where $g_{0}=dx^{2}$ is the standard metric on $\mathbb R^n$ and $k\ (\neq{1}) $ is a positive constant.

Let $(N,h)\hookrightarrow\mathbb{R}^{k}$ be a $l$-dimensional, smooth
compact Riemannian manifold without boundary, isometrically embedded in the Euclidean
space $\mathbb R^k$.

Motivated by the recent studies on elliptic systems in domains consisting of composite materials
(see Li-Nirenberg \cite{ln}) and the homogenization theory in calculus of variations (see Avellaneda-Lin
\cite{al} and Lin-Yan \cite{LY}), we are interested in the regularity issue of stationary harmonic maps from $(\Omega,g)$ to $(N,h)$.

In order to describe the problem, let's first recall some notations. Throughout this paper, we
use the Einstein convention for summation.
For the metric $g=g_{ij}\,dx^i\,dx^j$, let $(g^{ij})$ denote the inverse matrix of $(g_{ij})$, and
$dv_g=\sqrt{g}\,dx=\sqrt{\det{(g_{ij})}}\,dx$ denote the volume form of $g$.
For $1<p<+\infty$, define the Sobolev space 
$$W^{1,p}(\Omega,N)=\left\{u:\Omega\rightarrow \mathbb R^k
\ \bigg| ~u(x)\in N \ {\rm{a.e.}}\ x\in\Omega, \
~E_p(u,g)=\int_{\Omega} |\nabla u|_g^{p}\,dv_g<+\infty \right\},$$
where
$$\left|\nabla u\right|_g^2\equiv g^{ij}\langle\frac{\partial{u}}{\partial{x}_{i}},
\frac{\partial{u}}{\partial{x}_{j}}\rangle$$
is the $L^2$-energy density of $u$ with respect to $g$, and
$\langle\cdot,\cdot\rangle$ denotes the inner product in $\mathbb R^k$.
Denote $W^{1,2}(\Omega,N)$ by $H^1(\Omega,N)$.

Now let's recall the concept of stationary harmonic maps.
\begin{definition} A map $u\in{H}^{1}(\Omega,N)$ is called a (weakly) harmonic map,
if it is a critical point of $E_2(\cdot,g)$, i.e., $u$ satisfies
\begin{equation}\label{harmonic}
\Delta_{g}u+A(u)(\nabla{u},\nabla{u})_{g}=0
\end{equation}
in the sense of distributions. Here
$$\Delta_{g}=\frac{1}{\sqrt{g}}\frac{\partial}{\partial{x}_{i}}\left(\sqrt{g}g^{ij}\frac{\partial}{\partial{x}_{j}}\right)$$
is the Laplace-Beltrami operator on $(\Omega, g)$,
$A(\cdot)(\cdot,\cdot)$ is the second fundamental form of $(N,h)\hookrightarrow\mathbb{R}^{k}$, and $$A(u)(\nabla{u},\nabla{u})_{g}
=g^{ij}A(u)\left(\frac{\partial{u}}{\partial{x}_{i}},\frac{\partial{u}}{\partial{x}_{j}}\right).$$
\end{definition}

\begin{definition} A (weakly) harmonic map $u\in H^1(\Omega, N)$
is called a stationary harmonic map,  if, in additions, it is a critical point of
$E_2(\cdot,g)$ with respect to suitable domain variations:
\begin{equation}\label{stationary-harmonic}
\frac{d}{dt}\bigg|_{t=0}\int_{\Omega}\big|\nabla{u}^{t}\big|^{2}_{g}\ dv_{g}=0,
\ \ {\rm{with}}\ \  u^{t}(x)=u(F_t(x)),
\end{equation}
where $\displaystyle F(t,x):=F_t(x)\in C^1([-\delta,\delta], C^1(\Omega,\Omega))$ is a $C^1$ family of differmorphisms
for some small $\delta>0$ satisfying
\begin{equation}\label{suitable}
\begin{cases}
F_0(x)=x  & \forall x\in\Omega,\\
F_t(x)=x & \forall (x,t)\in\partial\Omega\times [-\delta,\delta],\\
F_t\Big(\overline{\Omega^\pm}\Big)\subset\overline{\Omega^\pm} & \forall t\in [-\delta,\delta].
\end{cases}
\end{equation}
\end{definition}
It is readily seen that any minimizing harmonic map from $(\Omega, g)$ to $(N,h)$ is a stationary harmonic
map.  It is also easy to see from Definition 1.2 that a stationary harmonic map on $(\Omega, g)$
is a stationary harmonic map on $(\Omega^\pm,g)$ and hence satisfies an energy monotonicity
inequality on $\Omega^\pm$, since $g\in C^{0,1}(\Omega^\pm)$. We will show in \S2 that a stationary
harmonic map on $(\Omega, g)$ also satisfies an energy monotonicity inequality in $\Omega$
under the condition (\ref{jump_cond}) below.

The first result is concerned with both the (partial) Lipschitz regularity and
(partial) piecewise $C^{1,\alpha}$-regularity of stationary harmonic maps.
In this context, we are able to extend the well-known partial regularity
theorem of stationary harmonic maps on domains with smooth metrics, due
to H\'elein \cite{h},  Evans \cite{e}, Bethuel \cite{b}. More precisely,
we have

\begin{theorem}\label{thm1}
Let $u\in{H}^{1}(\Omega,N)$ be a stationary harmonic map on $(\Omega, g)$. If, in additions,
$g$ satisfies the following
jump condition on $\Gamma$ for $n\ge 3$ \footnote{This condition is needed
for both energy monotonicity inequalities for $u$ in dimensions $n\ge 3$ and the piecewise $C^{1,\alpha}$-regularity of $u$.}\ : for any $x\in \Gamma$, there exists a positive constant
$k(x)\not=1$ such that
\begin{equation}\label{jump_cond}
\lim_{y\in\Omega^+, y\rightarrow x} g(y)
=k(x)\lim_{y\in\Omega^-, y\rightarrow x} g(y),
\end{equation}
then
there exists a closed set $\Sigma\subset\Omega$, with $H^{n-2}(\Sigma)=0$, such that
for some $0<\alpha<1$,
$$(\mathrm{i})~u\in{\mathrm{Lip}}_{\mathrm{loc}}(\Omega\setminus\Sigma,N),
\quad\quad(\mathrm{ii})~u\in{C}_{\mathrm{loc}}^{1,\alpha}((\Omega^{+}\cup\Gamma)\setminus\Sigma,N)
\cap {C}_{\mathrm{loc}}^{1,\alpha}((\Omega^{-}\cup\Gamma)\setminus\Sigma,N).$$
\end{theorem}

We would like to remark that when the dimension $n=2$, since the energy monotonicity inequality automatically holds for $H^1$-maps,
Theorem \ref{thm1} holds for any weakly harmonic map from domains
of piecewise $C^{0,1}$-metrics, i.e., any weakly harmonic map on domains with the above piecewise Lipschitz continuous metrics is both Lipschitz continuous and piecewise $C^{1,\alpha}$ for some $0<\alpha<1$.

Through the example constructed by Rivi\`ere \cite{R}, we know that weakly harmonic
maps on domains with smooth metrics may not enjoy partial regularity properties in dimensions $n\ge 3$.
Here we consider weakly harmonic maps on domains with piecewise
Lipschitz continuous metrics into any Riemannian manifold $(N,h)$,
on which $d_N^2(\cdot, p)$ is convex. Such Riemannian manifolds $N$
include those with non-positive sectional curvatures, and geodesic
convex balls in any Riemannian manifold. In particular, we 
extend the classical regularity theorems on harmonic maps on domains with
smooth metrics, due to Eells-Sampson \cite{ES} and Hildebrandt-Kaul-Widman \cite{HKW},
and prove

\begin{theorem}\label{thm2} Let
$g$ be the same as in Theorem \ref{thm1}.
Assume that on the universal cover $(\widetilde{N},\widetilde h)$ of
$(N,h)$\footnote{Here the covering map $\Pi:\widetilde N\to N$ is a Riemannian submersion from
($\widetilde {N},\widetilde{h}$) to ($N,h)$.},
 the square of distance function $d^2_{\widetilde N}(\cdot,p)$ is convex for any $p\in\widetilde N$.
If $u\in H^1(\Omega, N)$ is a weakly harmonic map, then for some $0<\alpha<1$,
$$(\mathrm{i})~u\in\mathrm{Lip}_{\rm{loc}}(\Omega,N),
\quad\quad\quad(\mathrm{ii})~u\in{C}^{1,\alpha}_{\rm{loc}}(\Omega^{+}\cup\Gamma,N)
\cap {C}^{1,\alpha}_{\rm{loc}}(\Omega^{-}\cup\Gamma,N) .$$
\end{theorem}

The idea to prove Theorem 1.1 is motivated by Evans \cite{e}
and Bethuel \cite{b}. However, there are several new
difficulties that we have to overcome. The first difficulty is to establish an almost energy monotonicity
inequality for stationary harmonic maps in $\Omega$, which is achieved by
observing that an exact monotonicity inequality holds at any $x\in\Gamma$, see \S2 below.
The second one is to establish a Hodge decomposition in $L^p(B,\mathbb R^n)$, for any $1<p<+\infty$, on a ball $B  (=B_r(0))$ equipped with
certain piecewise continuous metrics $g$, 
in order to adapt the argument by Bethuel \cite{b}.  More precisely, we
will show that the following elliptic equation on $B$:
$$\begin{cases}\frac{\partial}{\partial x_i}(a_{ij}\frac{\partial v}{\partial x_j})
=\hbox{div}(f) & \ {\rm{in}}\ B,\\
v=0 & \ {\rm{on}}\ \partial B
\end{cases}
$$
enjoys the $W^{1,p}$-estimate: for any $1<p<+\infty$,
$$\Big\|\nabla v\Big\|_{L^p(B)}\leq C \Big\|f\Big\|_{L^p(B)}$$
provided that $\displaystyle (a_{ij})\in C\left(\overline{B^\pm}\right) \cap C\left(B^\delta\right)$ for some $\delta>0$ is uniformly elliptic, and is discontinuous
on $\partial B^+\setminus B^\delta$, where
$\displaystyle B^\delta=\Big\{x\in B: {\rm{dist}}(x,\partial B)\le\delta\Big\}$.

This fact follows from a recent Theorem by Byun-Wang \cite{bw}, see \S3 below.
The third one is to employ the moving frame method to establish a decay estimate in suitable Morrey spaces under a smallness condition, which
is similar to \cite{iw}. To obtain Lipschitz and piecewise $C^{1,\alpha}$-regularity, 
we compare the harmonic map system with an elliptic system with piecewise
constant coefficients and extend the hole-filling argument by Giaquinta-Hildebrandt
\cite{GH}.

The paper is organized as follows. In \S 2, we derive an almost energy
monotonicity inequality. In \S 3,
we show the global $W^{1,p}$ $(1<p<\infty)$ estimate for
elliptic systems with certain piecewise continuous coefficients, and 
a Hodge decomposition theorem.
In \S 4, we adapt the moving frame method, due to H\'{e}lein \cite{h}
and Bethuel \cite{b}, to establish an $\epsilon$-H\"older continuity.
In \S 5, we establish both Lipschitz and piecewise $C^{1,\alpha}$ regularity for
H\"older continuous harmonic maps. In \S 6, we consider harmonic maps into manifolds supporting
convex distance functions and prove Theorem
\ref{thm2}.

\medskip
\noindent{\bf Acknowledgement}. Part of this work was completed while the first author visited University of
Kentucky. He would like to thank the Department of Mathematics for its hospitality. 
The first author was partially supported by SRFDPHE (20100003120005)
and NNSF in China (11071020) and Program for Changjiang Scholars and
Innovative Research Team in University in China.
The second author is partially supported by NSF grant 1000115.

\section{Energy monotonicity inequality}

This section is devoted to the derivation of energy monotonicity inequalities for stationary harmonic maps
from $(\Omega,g)$ to $(N,h)$. More precisely, we have
\begin{theorem} \label{thm21} Under the same assumption as in Theorem \ref{thm1},
there exist $C>0$ and $r_0>0$ depending only on $\Gamma$ and  $g$ such that
if $u\in W^{1,2}(\Omega, N)$ is a stationary harmonic map on $(\Omega,g)$,
then for any $x_0\in \Omega$, there holds
\begin{equation}\label{monotonicity_ineq}
s^{2-n}\int_{B_s(x_0)}\big|\nabla u\big|_g^2\,dv_g
\le e^{Cr} r^{2-n}\int_{B_r(x_0)}\big|\nabla u\big|_g^2\,dv_g
\end{equation}
for all $\displaystyle 0<s\le r\le \min\{r_0, {\rm{dist}}(x_0,\partial\Omega)\}$.
\end{theorem}

Since the metric $g\in C^{0,1}(\Omega^\pm)$,  it is well-known that there are $K>0$ and $r_0>0$ such that (\ref{monotonicity_ineq})
holds for any $x_0\in\Omega^\pm$ and $\displaystyle 0<s\le r\le \min\{r_0, {\rm{dist}}(x_0,\partial\Omega^\pm)\}$, see \cite{h}. In particular, (\ref{monotonicity_ineq}) holds for any $\displaystyle x_0\in\Omega\setminus \Gamma^{r_0}$ and $\displaystyle 0<s\le r\le\min\{r_0, {\rm{dist}}(x_0,\partial\Omega)\}$, where $\displaystyle\Gamma^{r_0}=\{x\in\Omega: \ {\rm{dist}}(x,\Gamma)\le r_0\}$
is the $r_0$-neighborhood of $\Gamma$.
We will see that to show (\ref{monotonicity_ineq}) for $x_0\in\Gamma^{r_0}$, it suffices
to consider the case $x_0\in\Gamma$.

It follows from the assumption on $\Gamma$ and $g$, there exists $r_0>0$ such that for
any $x_0\in\Gamma$ there exists a $C^{1,1}$-differmorphism $\Phi_{0}:B_1\to B_{r_1}(x_0)$,
where $\displaystyle r_1=\min\{r_0, {\rm{dist}}(x_0,\partial\Omega)\}$, such that
$$\begin{cases}
\Phi_{0}(B_1^\pm)=\Omega^\pm\cap B_{r_1}(x_0)\\
\Phi_0(\Gamma_1)=\Gamma\cap B_{r_1}(x_0), \ \mbox{where}\ \Gamma_1 =\{x\in B_1: x_n=0\}.
\end{cases}
$$
Define $\displaystyle\widetilde u(x)=u(\Phi_0(x))$ and
$\displaystyle\widetilde g(x)=(\Phi_0)_*(g)(x)$ for $ x\in B_1$.  Then it is
readily seen that\\
(i) $\widetilde g$ is piecewise $C^{0,1}$, with
the discontinuous set $\Gamma_1$, and satisfies (\ref{jump_cond}) on $\Gamma_1$
\footnote{In fact, since
$\displaystyle (\Phi_0)_*(g)_{ij}(x)
=g_{kl}(\Phi_0(x))\frac{\partial\Phi_0^k}{\partial x_i}(x)\frac{\partial\Phi_0^l}{\partial x_j}(x),$
(\ref{jump_cond}) implies that for any $x\in\Gamma_1$
$$\lim_{y\in\Omega^+,y\rightarrow x}(\Phi_0)_*g(y)
=k(\Phi_0(x))\lim_{y\in\Omega^-,y\rightarrow x}(\Phi_0)_*g(y).$$}, \\
(ii)  If $\displaystyle u: (B_{r_1}(x_0),\ g)\to (N,\ h)$
is a stationary harmonic map, so does
$\displaystyle\widetilde u: (B_1,\ \widetilde g)\to (N,\ h)$.

Thus we may assume that $\Omega=B_1$, $g$ is a piecewise $C^{0,1}$-metric which satisfies
(\ref{jump_cond}) on the set of discontinuity $\Gamma_1$, and $u: (B_1,g)\to (N,h)$
is a stationary harmonic map. It suffices to establish (\ref{monotonicity_ineq})
in $B_{\frac12}$. We first derive a stationarity identity for $u$.
\begin{prop}
Let $u\in W^{1,2}(B_{1},N)$ be a stationary harmonic map on $(B_1, g)$. Then
\begin{equation}\label{prop21.0}
\int_{B_{1}}\left(2g^{ij}\langle\frac{\partial{u}}{\partial{x}_{k}},\frac{\partial{u}}{\partial{x}_{j}}\rangle Y^k_i-|\nabla{u}|^{2}_{g}\mathrm{div}Y\right)\sqrt{g}\,dx
=\int_{B_1}\frac{\partial}{\partial x_k}\Big(\sqrt{g}g^{ij}\Big) Y^k \langle\frac{\partial{u}}{\partial{x}_i},\frac{\partial{u}}{\partial{x}_{j}}\rangle\,dx
\end{equation}
holds for all $\displaystyle Y=(Y^1,\cdots, Y^{n-1}, Y^n)\in C_{0}^{1}(B_1,\mathbb{R}^{n})$ satisfying
\begin{equation}\label{specific}
Y^n(x)\begin{cases} \ge 0& \mbox{for}\ x^n> 0\\
 =0& \mbox{for}\ x^n=0\\
\le 0& \mbox{for}\ x^n<0,
\end{cases}
\end{equation}
where $\displaystyle Y^k_i=\frac{\partial Y^k}{\partial x_i}$ and
$\displaystyle\mathrm{div} \ Y=\sum_{i=1}^n \frac{\partial Y^i}{\partial x_i}$.
\end{prop}
\begin{proof} Let $Y$ satisfy (\ref{specific}), it is easy to see that there exists $\delta>0$ such that
$F_t(x)=x+tY(x),\ t\in [-\delta,\delta]$, is a family of differmorphisms from $B_1$ to $B_1$ satisfying
the condition (1.3). Hence
$$0=\frac{d}{dt}\Big|_{t=0}\int_{B_1}|\nabla (u(F_t(x))|_g^2\,dv_g
=\frac{d}{dt}\Big|_{t=0}\Big(\int_{B_1^+}|\nabla (u(F_t(x))|_g^2\,dv_g
+\int_{B_1^-}|\nabla (u(F_t(x))|_g^2\,dv_g\Big).
$$
For $t\in [-\delta, \delta]$, set $G_t=F_t^{-1}$. Direct calculations yield
\begin{eqnarray*}
&&\frac{d}{dt}\Big|_{t=0}\int_{B_1^\pm}|\nabla (u(F_t(x))|_g^2\,dv_g\\
&=&\frac{d}{dt}\Big|_{t=0}\int_{B_1^\pm}\sqrt{g(x)}g^{ij}(x)\langle\frac{\partial u}{\partial y_k},
\frac{\partial u}{\partial y_l}\rangle(x+tY(x))
(\delta_{ki}+tY^k_i)(\delta_{lj}+tY^l_j)\,dx\\
&=&\int_{B_1^\pm}\sqrt{g} g^{ij}\langle\frac{\partial u}{\partial x_k},
\frac{\partial u}{\partial x_l}\rangle(\delta_{ki}Y^l_j+\delta_{lj} Y^k_i)\,dx\\
&&+\int_{B_1^\pm}\frac{d}{dt}\Big|_{t=0}\left(g^{ij}(G_t(x))\sqrt{g(G_t(x))}JG_t(x)\right)
\langle\frac{\partial u}{\partial x_i},
\frac{\partial u}{\partial x_j}\rangle\,dx\\
&=&\int_{B_1^\pm} \Big(2g^{ij}\langle\frac{\partial u}{\partial x_i},
\frac{\partial u}{\partial x_l} \rangle Y^l_j-g^{ij}\langle\frac{\partial u}{\partial x_i},
\frac{\partial u}{\partial x_j}\rangle{\rm{div}} Y\Big)\sqrt{g}\,dx\\
&&-\int_{B_1^\pm}\frac{\partial}{\partial x_k}\Big(\sqrt{g}g^{ij}\Big) Y^k \langle\frac{\partial{u}}{\partial{x}_i},\frac{\partial{u}}{\partial{x}_{j}}\rangle\,dx,
\end{eqnarray*}
where we have used
$$\begin{cases}\frac{d}{dt}\Big|_{t=0}JG_t(x)=-{\rm{div}} Y,\\
\frac{d}{dt}\Big|_{t=0} G_t(x)=-Y(x),\\
\frac{d}{dt}\Big|_{t=0}\Big(g^{ij}(G_t(x))\sqrt{g(G_t(x))}\Big)
=-\frac{\partial}{\partial x_k}\Big(\sqrt{g}g^{ij}\Big) Y^k.
\end{cases} 
$$
This completes the proof.
\end{proof}

\begin{prop}\label{cor1}
Let $u\in{W}^{1,2}(B_{1},N)$ be a stationary harmonic map on $(B_1,g)$. Then
there exists $C>0$ such that
\begin{itemize}
\item[(i)] for any $x^{0}=(x_{0}^{'},x_{0}^{n})\in{B}_{\frac12}\setminus\Gamma_1$, there exists $0<R_{0}\leq\min\{\frac{1}{4},|x_{0}^{n}|\}$, such that
    \begin{equation}\label{cor1.01}
    r^{2-n}\int_{B_{r}(x^{0})}|\nabla{u}|_{g}^{2}dv_{g}\leq\, e^{CR} R^{2-n}\int_{B_{R}(x^{0})}|\nabla{u}|_{g}^{2}dv_{g},
    \quad\quad 0<r\leq{R}<R_{0}.
    \end{equation}
\item[(ii)] for any $x^{0}\in{B}_{\frac{1}{2}}\cap\Gamma_1$, there holds
    \begin{equation}\label{cor1.02}
    r^{2-n}\int_{B_{r}(x^{0})}|\nabla{u}|_{g}^{2}dv_{g}\leq\, e^{CR} R^{2-n}\int_{B_{R}(x^{0})}|\nabla{u}|_{g}^{2}dv_{g},
    \quad\quad 0<r\leq{R}\le\frac{1}{4}.
    \end{equation}
\end{itemize}
In particular, for any $x^{0}\in{B}_{\frac{1}{2}}$, there holds
\begin{equation}\label{cor2.0}
    r^{2-n}\int_{B_{r}(x^{0})}|\nabla{u}|_{g}^{2}dv_{g}\leq\, e^{CR} R^{2-n}\int_{B_{R}(x^{0})}|\nabla{u}|_{g}^{2}dv_{g},
\ \ \ 0<r\leq{R}\le\frac14.
\end{equation}
\end{prop}
\begin{proof}

(i) By choosing $Y\in{C}^{\infty}_{c}(B_{1}^{+},\mathbb{R}^{n})$ or
$Y\in{C}^{\infty}_{c}(B_{1}^{-},\mathbb{R}^{n})$, we have
that $u$ is a stationary harmonic map on $(\,B_{1}^{+},g)$ and
$(\,B_{1}^{-}, g)$. Thus the monotonicity inequality (\ref{cor1.01}) 
is standard, see \cite{h}.

(ii) For simplicity, consider $x^{0}=(0',0)$. For $\epsilon>0$ and $0<r\le \frac12$, let
$Y_\epsilon(x)=x\eta_\epsilon(x)$, where $\eta_\epsilon(x)=\eta_\epsilon(|x|)\in C_0^\infty(B_1)$ 
satisfies
$$0\le\eta_\epsilon\le 1; \ \eta_\epsilon(s)\equiv 1\ {\rm{for}}\ 0\le s\le {r-\epsilon};
\ \eta_\epsilon(s)\equiv 0 \ {\rm{for}}\ s\ge r; \eta_\epsilon'\le 0;
\ |\eta_\epsilon'|\le \frac{2}{\epsilon}.$$
Then
\begin{equation}\label{deform_vector}
(Y_\epsilon)^j_i=\delta_{ij}\eta_\epsilon(|x|)+\eta_\epsilon'(|x|)\frac{x^{i}x^{j}}{|x|}.
\end{equation}
Substituting $Y_\epsilon$ into the right hand side of (\ref{prop21.0}), and using
$$\Big|\frac{\partial}{\partial x_k}\Big(\sqrt{g}g^{ij}\Big)\Big|
\le C,$$
we have
\begin{equation}\label{rhs}
\Big|\int_{B_1}\frac{\partial}{\partial x_k}\Big(\sqrt{g}g^{ij}\Big) Y_\epsilon^k \langle\frac{\partial{u}}{\partial{x}_i},\frac{\partial{u}}{\partial{x}_{j}}\rangle\,dx\Big|
\le Cr\int_{B_r}|\nabla u|^2\,dx\le Cr\int_{B_r}|\nabla u|^2_g\,dv_g.
\end{equation}
Substituting (\ref{deform_vector}) into the left hand side of (\ref{prop21.0}),
we obtain
\begin{align}\label{cor1.1}
&\int_{B_1}\Big(2g^{ij}\langle \frac{\partial u}{\partial x_j}, \frac{\partial u}{\partial x_k}\rangle
(Y_\epsilon)_i^k-|\nabla u|^2_g {\rm{div}} Y_\epsilon\Big)\sqrt{g}\,dx\nonumber\\
&=(2-n)\int_{B_1}|\nabla u|_g^2\eta_\epsilon(x) \sqrt{g}\,dx
-\int_{B_1}|\nabla u|^2_g |x|\eta_\epsilon'(x)\sqrt{g}\,dx\nonumber\\
&\ +\int_{B_1} 2g^{ij}\langle \frac{\partial u}{\partial x_i}, \frac{\partial u}{\partial x_k}\rangle \frac{x^k x^j}{|x|}\eta_\epsilon'(x)\sqrt{g}\,dx.
\end{align}
Set the piecewise constant metric $\overline g$ by
$$
\overline g(x',x^n)=\begin{cases} \lim\limits_{y\rightarrow 0, \ y^n\ge 0} g(y) & \ {\rm{if}}\ x^n\ge 0\\
\lim\limits_{y\rightarrow 0, \ y^n< 0} g(y) & \ {\rm{if}}\ x^n<0.
\end{cases}
$$
Then we have
\begin{equation}\label{lip}
|g(x)-\overline g(x)|\le C|x|, \ \forall x\in B_1.
\end{equation}
It follows from  (\ref{jump_cond}) that we can assume 
$$\overline g(x)=\begin{cases} g_0 & \ {\rm{if}}\ x^n\ge 0\\
kg_0 & \ {\rm{if}}\ x^n<0,
\end{cases}
$$
for some positive constant $k\not=1$.
Thus we can estimate
\begin{eqnarray}\label{lhs2}
&&\int_{B_1} 2g^{ij}\langle \frac{\partial u}{\partial x_i}, \frac{\partial u}{\partial x_k}\rangle \frac{x^k x^j}{|x|}\eta_\epsilon'(x)\sqrt{g}\,dx\nonumber\\
&= &2\int_{B_1}\overline{g}^{ij}\langle \frac{\partial u}{\partial x_i}, \frac{\partial u}{\partial x_k}\rangle \frac{x^k x^j}{|x|}\eta_\epsilon'(x)\sqrt{g}\,dx
+2\int_{B_1}(g^{ij}-\overline{g}^{ij})\langle \frac{\partial u}{\partial x_i}, \frac{\partial u}{\partial x_k}\rangle \frac{x^k x^j}{|x|}\eta_\epsilon'(x)\sqrt{g}\,dx\nonumber\\
&=&I_\epsilon+II_\epsilon.
\end{eqnarray}
Since
$$\overline{g}^{ij}\langle \frac{\partial u}{\partial x_i}, \frac{\partial u}{\partial x_k}\rangle \frac{x^k x^j}{|x|}
\equiv h(x):=\begin{cases} |x||\frac{\partial u}{\partial r}|^2 & \ {\rm{if}}\ x^n\ge 0\\
 \frac1{k}|x||\frac{\partial u}{\partial r}|^2 & \ {\rm{if}}\ x^n< 0,
\end{cases}
$$
and $h(x)\ge 0$ for $x\in B_1$, we have
\begin{equation}\label{1epsilon}
I_\epsilon=\int_{B_1}h(x)\eta_\epsilon'(|x|)\sqrt{g}\,dx\le 0.
\end{equation}
For $II_\epsilon$, by (\ref{lip})  we have
\begin{equation}\label{2epsilon}\Big|II_\epsilon\Big|
\le Cr\int_{B_r}|\nabla u|^2\,dv_g \le Cr\int_{B_r}|\nabla u|^2_g\,dv_g.
\end{equation}
First substituting (\ref{1epsilon}) and (\ref{2epsilon}) into (\ref{lhs2}),  and then plugging the resulting (\ref{lhs2})  into (\ref{cor1.1}), and finally combining (\ref{cor1.1}) and (\ref{rhs}) with (\ref{prop21.0}), we obtain, after sending $\epsilon$ to zero,
$$(2-n)\int_{B_{r}}|\nabla{u}|_{g}^{2}dv_{g}+r\int_{\partial{B}_{r}}\left|\nabla u\right|_{g}^{2}\sqrt{g}\,dH^{n-1}
\ge -Cr\int_{{B}_{r}}|\nabla{u}|_{g}^{2}dv_{g}.$$
This implies
$$\frac{d}{dr}\left(e^{Cr}r^{2-n}\int_{B_{r}}|\nabla{u}|^{2}_{g}dv_{g}\right)\ge 0,$$
which clearly yields \eqref{cor1.02}.

To show (\ref{cor2.0}), it suffices to consider the case 
$$x^{0}\in{B}_{1/2}\setminus\Gamma_1,
\quad|B_{R}(x^{0})\cap{B}_{1}^{+}|>0\ \mbox{and} \ |B_{R}(x^{0})\cap{B}_{1}^{-}|>0.$$
For simplicity, assume $x^{0}\in{B}_{1}^{-}$. We divide it into two cases:\\
(i) $d(x^{0},\Gamma_1)=|x^{0}_{n}|\geq\frac{1}{4}R$:
\begin{itemize}
\item If $R\geq{r}\geq\frac{1}{4}R$, then it is easy to see
$$r^{2-n}\int_{B_{r}(x^{0})}|\nabla{u}|_{g}^{2}dv_{g}\leq\,4^{n-2}R^{2-n}\int_{B_{R}(x^{0})}|\nabla{u}|_{g}^{2}dv_{g}.$$
\item If $0<r<\frac{1}{4}R (\le d(x^{0},\Gamma_1))$, we have
$B_{\frac{R}{4}}(x^{0})\subset{B}_{1}^{-}$ so that (\ref{cor1.01}) implies
$$r^{2-n}\int_{B_{r}(x^{0})}|\nabla{u}|_{g}^{2}dv_{g}\leq e^{CR}\left(\frac{R}{4}\right)^{2-n}\int_{B_{\frac{R}{4}}(x^{0})}|\nabla{u}|_{g}^{2}dv_{g}
\leq\, e^{CR}R^{2-n}\int_{B_{R}(x^{0})}|\nabla{u}|_{g}^{2}dv_{g}.$$
\end{itemize}
(ii) $d(x^{0},\Gamma_1)=|x^{0}_{n}|<\frac{1}{4}R$:
\begin{itemize}
\item
If $R\geq{r}\geq\frac{1}{4}R$, then
$$r^{2-n}\int_{B_{r}(x^{0})}|\nabla{u}|_{g}^{2}dv_{g}\leq\,4^{n-2}R^{2-n}\int_{B_{R}(x^{0})}|\nabla{u}|_{g}^{2}dv_{g}.$$
\item
If $0<r\leq\,d(x^{0},\Gamma_1)=|x_n^0|<\frac{1}{4}R$, then by setting
$\overline{x}^{0}=(x^{0}_{1},\cdots,x_{n-1}^{0},0)$ we have
$$B_{r}(x^{0})\subset\,B_{|x^{0}_{n}|}(x^{0})\subset\,B_{2|x_{n}^{0}|}(\overline{x}^{0})\subset\,B_{\frac{R}{2}}(\overline{x}^{0})\subset\,B_{R}(x^{0}),$$
so that (\ref{cor1.02}) yields
\begin{align*}
r^{2-n}\int_{B_{r}(x^{0})}|\nabla{u}|_{g}^{2}dv_{g}&\leq|x_{n}^{0}|^{2-n}\int_{B_{|x_{n}^{0}|}(x^{0})}|\nabla{u}|_{g}^{2}dv_{g}\\
&\leq2^{n-2}(2|x_{n}^{0}|)^{2-n}\int_{B_{2|x_{n}^{0}|}(\overline{x}^{0})}|\nabla{u}|_{g}^{2}dv_{g}\\
&\leq2^{n-2}e^{CR}\left(\frac{R}{2}\right)^{2-n}\int_{B_{\frac{R}{2}}(\overline{x}^{0})}|\nabla{u}|_{g}^{2}dv_{g}\\
&\leq e^{CR}R^{2-n}\int_{B_{R}(x^{0})}|\nabla{u}|_{g}^{2}dv_{g}.
\end{align*}
\item
If $d(x^{0},\Gamma_1) (=|x_n^0|)\leq\,r<\frac{1}{4}R$, then we have
$$B_{r}(x^{0})\subset\,B_{2r}(\overline{x}^{0})\subset\,B_{\frac{R}{2}}(\overline{x}^{0})\subset\,B_{R}(x^{0}),$$
so that (\ref{cor1.02}) yields
\begin{align*}
r^{2-n}\int_{B_{r}(x^{0})}|\nabla{u}|_{g}^{2}dv_{g}&\leq2^{n-2}(2r)^{2-n}\int_{B_{2r}(\overline{x}^{0})}|\nabla{u}|_{g}^{2}dv_{g}\\
&\leq2^{n-2}e^{CR}\left(\frac{R}{2}\right)^{2-n}\int_{B_{\frac{R}{2}}(\overline{x}^{0})}|\nabla{u}|_{g}^{2}dv_{g}\\
&\leq e^{CR}R^{2-n}\int_{B_{R}(x^{0})}|\nabla{u}|_{g}^{2}dv_{g}.
\end{align*}
\end{itemize}
Therefore (\ref{cor2.0}) is proven.
\end{proof}

\section{$W^{1,p}$-estimate for elliptic equations with certain piecewise continuous coefficients}

In this section, we will show the global $W^{1,p}$-estimate for
elliptic equations with certain piecewise continuous coefficients, for
$1<p<+\infty$. As a corollary, we will establish the Hodge decomposition Theorem
3.2 for certain piecewise continuous metrics $g$, which is a key ingredient to prove Theorem \ref{thm1}
and may also have its own interest.

For a ball $B=B_r(0)\subset\mathbb R^n$, denote $B^\epsilon=\{x\in B: \ {\rm{dist}}(x,\partial B)\le\epsilon\}$  for $\epsilon>0$.
Let $\displaystyle (a_{ij}(x))_{1\le i, j\le n}$
be bounded measurable, uniformly elliptic on $B$, i.e., there exists $0<\lambda\le\Lambda<+\infty$ such that
\begin{equation}\label{rank1-convex}
\lambda|\xi|^2\le a_{ij}(x)\xi_i^\alpha\xi_\beta^j\le \Lambda |\xi|^2,
\ \ {\rm{a.e.}}\ x\in B, \ \forall \xi\in \mathbb R^{n}.
\end{equation}

\begin{theorem}\label{lem31} Assume $\displaystyle (a_{ij})$ satisfies
(\ref{rank1-convex}), and there exists $\epsilon>0$ such that
$\displaystyle (a_{ij})\in C\left(\overline{B^\pm}\right)\cap C\left(B^\epsilon\right)$ and is discontinuous on  $\partial B^+\setminus B^{\epsilon}$.
For $1<p<+\infty$, let
$f\in{L}^{p}(B,\mathbb{R}^{n})$. Then there exists a unique weak solution
$\displaystyle v\in W^{1,p}_0(B, \mathbb{R}^{n})$ to
\begin{equation}\label{lem31.0}
\begin{cases}
\sum\limits_{i,j}\frac{\partial}{\partial x_i}\left(a_{ij}\frac{\partial v}{\partial x_j}\right)
=\sum\limits_{i}\frac{\partial f_i}{\partial x_i} &\mbox{in}~B,\\
u=0 &\mbox{on}~\partial{B},
\end{cases}
\end{equation}
and
\begin{equation}\label{lp-estimate}
\left\|\nabla{v}\right\|_{L^{p}(B)}\leq C\left\|f\right\|_{L^{p}(B)}
\end{equation}
for some $C>0$ depending only on $\displaystyle
p \ {\rm{and}}\  (a_{ij})$.
\end{theorem}
\begin{proof}
By our assumption, it is easy to verify that for any $\delta>0$,
there exists $R=R(\delta)>0$ such that the coefficient function $(a_{ij})$ satisfies the  $(\delta,R)$-vanishing of  codimension 1
conditions (2.5) and (2.6) of Byun-Wang \cite{bw} page 2652. In fact, we have a stronger property:
$$
\lim_{r\downarrow 0}\max_{x_0=(x_0', x_0^n)\in\overline B}\Big\|a_{ij}(x',x^n)-a_{ij}(x_0', x^n)\Big\|_{L^\infty\left(B_r((x_0', x_0^n))\right)}=0.
$$
Thus Theorem \ref{lem31} follows by direct applications of \cite{bw} Theorem 2.2, page  2653.
\end{proof}

As an immediate consequence of Theorem \ref{lem31}, we have the following Hodge decomposition on $B$ equipped with
certain piecewise continuous metrics $g$.
\begin{theorem}\label{thm36} Let
$\bar{g}$ be a piecewise continuous metric on $B$
such that $\bar g\in{C}\left({\overline{B^\pm}}\right)\cap C\left(B^\delta\right)$ for some $\delta>0$,
and is discontinuous on $\partial B^+\setminus B^\delta$.  Then for any $1<p<+\infty$, 
${F}=(F_{1},\cdots,F_{n})\in{L}^{p}(B,\mathbb{R}^{n})$, there
exist $G\in{W}_{0}^{1,p}(B)$ and
$H\in{L}^{p}(B,\mathbb{R}^{n})$ such that
\begin{equation}\label{thm36.01}
F=\nabla{G}+H,\ \ 0= \mathrm{div}_{\bar{g}}H\ (:=\frac{1}{\sqrt{\bar{g}}}\frac{\partial}{\partial{x}_{i}}(\sqrt{\bar{g}}\bar{g}^{ij}H_{j}))
\ {\rm{in}}\  B,
\end{equation}
and there exists $C=C(p,n,\bar g)>0$ such that
\begin{equation}\label{thm36.03}
\left\|\nabla{G}\right\|_{L^{p}(B)}+\left\|H\right\|_{L^{p}(B)}\leq C\left\|F\right\|_{L^{p}(B)}.
\end{equation}
\end{theorem}

\begin{proof} Set $a_{ij}=\sqrt{\bar g}{\bar g}^{ij}$ on $B$ for $1\le i, j\le n$. It is easy to verify that $(a_{ij})$
satisfies the conditions of Theorem \ref{lem31}. Thus Theorem \ref{lem31} yields that
there exists a unique solution $G\in{W}_{0}^{1,p}(B)$ to
\begin{equation}\label{thm36.1}
\begin{cases}
\frac{\partial}{\partial{x}_{i}}\left(\sqrt{\bar{g}}\bar{g}^{ij}\frac{\partial{G}}{\partial{x}_{j}}\right)
=\frac{\partial}{\partial{x}_{i}}\left(\sqrt{\bar{g}}\bar{g}^{ij}F_{j}\right) &\mbox{in}~B,\\
G=0 & \ {\rm{on}}\ \partial B,
\end{cases}
\end{equation}
and
$$\left\|\nabla{G}\right\|_{L^{p}(B)}\leq C\left\|\sqrt{\bar{g}}\bar{g}^{ij}F_{j}\right\|_{L^{p}(B)}
\leq C\left\|F\right\|_{L^{p}(B)}.$$
Set $H=F-\nabla{G}$.  Then we have
$$\mathrm{div}_{\bar{g}}H=\frac{1}{\sqrt{\bar{g}}}\frac{\partial}{\partial{x}_{i}}
\left(\sqrt{\bar{g}}\bar{g}^{ij}\left(F_{j}-\frac{\partial{G}}{\partial{x}_{j}}\right)\right)=0 \ {\rm{on}}\ B,$$
and
$$\left\|H\right\|_{L^p(B_\frac12)}\le \left\|F\right\|_{L^p(B_\frac12)}+\left\|\nabla G\right\|_{L^p(B)}
\le C\left\|F\right\|_{L^p(B)}.$$
This completes the proof.
\end{proof}

\section{H\"{o}lder continuity}

In this section, we will prove that any stationary harmonic map on $(B_{1},g)$, with a piecewise Lipschitz
continuous metric $g\in{C}^{0,1}(B_{1}^{\pm}\cup\Gamma_1)$, is H\"older continuous under a smallness
condition of $\displaystyle\int_{B_1}|\nabla u|_g^2\,dv_g$. The idea is based on suitable modifications
of the original argument by Bethuel \cite{b} (see also Ishizuka-Wang \cite{iw}), thanks to the energy monotonicity inequality and the Hodge
decomposition theorem established in previous sections. More precisely, we have

\begin{theorem}\label{thm41}
There exist $\epsilon_{0}>0$ and $\alpha_{0}\in(0,1)$ depending only on $n, g$ such that if
the metric $g\in{C}^{0,1}(B_{1}^{\pm}\cup\Gamma_1)$ satisfies the condition (\ref{jump_cond})
on $\Gamma_1$, and
$u\in{W}^{1,2}(B_{1},N)$ is a stationary harmonic map on $(B_1,g)$ satisfying
\begin{equation}\label{small_norm_energy}
r_0^{2-n}\int_{B_{r_0}(x_0)}|\nabla{u}|_{g}^{2}\,dv_g\leq\epsilon_{0}^{2}
\end{equation}
for some $x_0\in{B}_{\frac{1}{2}}$ and $0<r_0\leq\frac{1}{4}$, then
$\displaystyle u\in{C}^{\alpha_{0}}(B_{\frac{r_0}{2}}(x_0),N)$ and
\begin{equation}\label{holder_continuity}
\Big[u\Big]_{C^{\alpha_{0}}(B_{\frac{r_0}{2}}(x_0))}\leq\,C(r_0, \epsilon_0).
\end{equation}
\end{theorem}

\begin{proof}[Proof of Theorem \ref{thm41}]
The proof is based on suitable modifications of \cite{b} and \cite{iw}.
First, observe that if $x_0=(x_0',x_0^{n})\in{B}^{\pm}$, it follows from the monotonicity inequality (\ref{cor2.0}) that we may assume  (\ref{small_norm_energy}) holds for some $0<r_0<|x_0^{n}|$.
Then the $\epsilon_{0}$-regularity theorem by Bethuel \cite{b} (see \cite{iw} for domains
with $C^{0,1}$ metrics)
implies that for some $0<\alpha_0<1$, $u\in{C}^{\alpha_0}(B_{\frac{r_0}2}(x_0))$ and (\ref{holder_continuity})
holds. Hence it suffices to consider the case 
$x_0=(x_0',0)\in\Gamma_\frac12$.  By translation and scaling, we may assume
$x_0=(0,0)$ and proceed as
follows.

\noindent{\it {Step 1.}} As in \cite{b} \cite{h} \cite{iw},
assume that there exists an orthonormal frame on $u^*TN\Big|_{B_1}$.
For $0<\theta<\frac{1}{2}$ to be determined later, let
$\{e_{\alpha}\}_{\alpha=1}^{l}\subset {W}^{1,2}(B_{2\theta},\mathbb{R}^{k})$ be
a Coulomb gauge orthonormal frame of $u^{*}TN\Big|_{B_{2\theta}}$:
\begin{equation}
\begin{cases}
\mbox{div}_{g}(\langle\nabla{e}_{\alpha},e_{\beta}\rangle)=0 \quad\mbox{in}~B_{2\theta} \quad(1\leq\alpha,\beta\leq\,l),\\
\sum\limits_{\alpha=1}^{l}\int_{B_{2\theta}}|\nabla{e}_{\alpha}|_{g}^{2}dv_{g}
\leq\,C\int_{B_{2\theta}}|\nabla{u}|_{g}^{2}dv_{g}.
\end{cases}
\end{equation}
For $1\le\alpha\le l$, consider $\displaystyle\langle\nabla\left((u-u_{2\theta})\eta\right),e_{\alpha}\rangle$,
where $\displaystyle u_{2\theta}=\fint_{B_{2\theta}}u$ is the average of $u$ on $B_{2\theta}$,  and
$\displaystyle\eta\in{C}_{0}^{\infty}(B_{1})$ satisfies
\begin{equation*}
0\leq\eta\leq1;\quad\eta=1~\mbox{in}~B_{\theta};\quad\eta=0~\mbox{outside}~B_{\frac74\theta};
\quad|\nabla\eta|\leq\frac{2}{\theta}.
\end{equation*}

Let $g_0$ be the standard metric on $\mathbb R^n$.  We define a new metric $\widetilde{g}$ on
$B_{2\theta}$ by letting
$$\widetilde{g}(x)=\eta(x)g(x)+(1-\eta(x))g_0(x), \ x\in B_{2\theta}.$$
Then it is easy to see that
$$\widetilde{g}\equiv g \ {\rm{on}}\ B_\theta,\
  \widetilde{g}\equiv g_0\ {\rm{outside}}\  B_{\frac74\theta},
\ {\rm{and}}\
\widetilde{g}\in C(\overline{B_{2\theta}^\pm})\cap C(B_{2\theta}\setminus B_{\frac74\theta}).$$
In particular, $\widetilde{g}$ satisfies the condition of Theorem \ref{thm36}.
Hence, by Theorem \ref{thm36}, we have that for $\displaystyle 1<p<\frac{n}{n-1}$, there exist
$\phi_{\alpha}\in{W}^{1,p}_0(B_{2\theta})$ and
$\psi_{\alpha}\in{L}^{p}(B_{2\theta})$ such that
\begin{equation}\label{hodge_decomp}
\begin{cases}
\langle\nabla\left((u-u_{2\theta})\eta\right),e_{\alpha}\rangle=\nabla\phi_{\alpha}+\psi_{\alpha},
\quad\mathrm{div}_{\widetilde{g}}(\psi_{\alpha})=0\ ~\mbox{in}~B_{2\theta}, \\
\|\nabla\phi_{\alpha}\|_{L^{p}(B_{2\theta})}+\|\psi_{\alpha}\|_{L^{p}(B_{2\theta})}
\lesssim\|\nabla\left((u-u_{2\theta})\eta\right)\|_{L^{p}(B_{2\theta})}\lesssim\|\nabla{u}\|_{L^{p}(B_{2\theta})}.
\end{cases}
\end{equation}
Since $u$ satisfies the harmonic map equation (\ref{harmonic}), we have
\begin{equation}\label{harmonic1}
\mathrm{div}_{g}\left(\langle\nabla{u},e_{\alpha}\rangle\right)=g^{ij}\nabla_{i}u\langle\nabla_{j}e_{\alpha},e_{\beta}\rangle\,e_{\beta}
\quad\mbox{in}~B_{1}.
\end{equation}
Thus we obtain
\begin{equation}
\label{harmonic2}
\Delta_{g}\phi_{\alpha}=g^{ij}\nabla_{i}u\langle\nabla_{j}e_{\alpha},e_{\beta}\rangle\,e_{\beta}
\quad\mbox{in}~B_{\theta}.
\end{equation}
Set
$\phi_{\alpha}=\phi_{\alpha}^{(1)}+\phi_{\alpha}^{(2)}$, where $\phi_\alpha^{(1)}$ solves
\begin{equation}\label{phi_1_eqn}
\begin{cases}
\Delta_{g}\phi_{\alpha}^{(1)}=0 &\mbox{in}~B_{\theta},\\
\phi_{\alpha}^{(1)}=\phi_{\alpha} &\mbox{on}~\partial{B}_{\theta},
\end{cases}
\end{equation}
and $\phi_{\alpha}^{(2)}$ solves
\begin{equation}\label{phi_2_eqn}
\begin{cases}
\Delta_{g}\phi_{\alpha}^{(2)}=g^{ij}\nabla_{i}u\langle\nabla_{j}e_{\alpha},e_{\beta}\rangle\,e_{\beta}
&\mbox{in}~B_{\theta},\\
\phi_{\alpha}^{(2)}=0 &\mbox{on}~\partial{B}_{\theta}.
\end{cases}
\end{equation}

\noindent{\it{Step 2.}} Estimation of $\phi_{\alpha}^{(1)}$: It is well-known (cf. \cite{gt})  that
$\phi_{\alpha}^{(1)}\in C^{\alpha_0}(B_{\theta})$ for some
$\alpha_0\in(0,1)$, and for any $0<r\le\frac{\theta}{2}$
\begin{equation}\label{phi_1_est1}\left[\phi_{\alpha}^{(1)}\right]^{p}_{C^{\alpha_0}(B_{\frac{r}{2}})}
\lesssim\theta^{p-n}\int_{B_{\theta}}|\nabla\phi_{\alpha}^{(1)}|^{p}\,dx\leq
C\theta^{p-n}\int_{B_{2\theta}}|\nabla{u}|^{p}\,dx,
\end{equation}
and
\begin{equation}\label{phi_1_est2}
(\tau\theta)^{p-n}\int_{B_{\tau\theta}}|\nabla\phi_{\alpha}^{(1)}|^{p}\leq
C\tau^{p\alpha_0}\Big\|\nabla{u}\Big\|_{M^{p,p}(B_1)}, \ \forall 0<\tau<1,
\end{equation}
where $M^{p,p} (\cdot)$ denotes the Morrey space: 
$$M^{p,p}(E)
:=\Big\{f:E\to\mathbb R: \left\|f\right\|^{p}_{M^{p,p}(E)}=\sup
_{B_r(x)\subset \mathbb R^n}\left\{r^{p-n}\int_{B_{r}(x)\cap E}|f|^{p}\,dx\right\}
<+\infty\Big\}, \ E\subset \mathbb R^n.$$
\noindent{\it{Step 3.}} Estimation of $\phi_{\alpha}^{(2)}$: First, denote by $\mathcal H^1(\mathbb R^n)$ 
the Hardy space on $\mathbb R^n$ and $\hbox{BMO}(E)$ the BMO space on $E$ for any open set $E\subset\mathbb R^n$.  By (4.13) of \cite{iw}  page 435,
for $\displaystyle p'=\frac{p}{p-1}>n$, there exists
$\displaystyle h\in{W}^{1,p'}_{0}(B_{\theta})$, with
$\displaystyle\|\nabla{h}\|_{L^{p'}(B_{\theta})}=1$, such that
$$\left\|\nabla\phi_{\alpha}^{(2)}\right\|_{L^{p}(B_{\theta})}\leq C\int_{B_{\theta}}\langle\nabla\phi_{\alpha}^{(2)},\nabla{h}\rangle_{g}dv_{g}.$$
Hence by the equation (\ref{phi_2_eqn}),  (\ref{hodge_decomp}), and the duality between $\mathcal H^1$
and BMO, we have
\begin{align}
\left\|\nabla\phi_{\alpha}^{(2)}\right\|_{L^{p}(B_{\theta})}
&\leq C\int_{B_{\theta}}\sqrt{g}g^{ij}\langle\nabla_{i}u\langle\nabla_{j}e_{\alpha},e_{\beta}\rangle\rangle(e_{\beta}h)
\,dx\nonumber\\
&=-C\int_{B_{\theta}}\sqrt{g}g^{ij}\langle\nabla_{j}e_{\alpha},e_{\beta}\rangle\rangle
\nabla_i(e_{\beta}h) u
\,dx\nonumber\\
&\leq C\Big\|\sqrt{g}g^{ij}\langle\nabla_{j}e_{\alpha},e_{\beta}\rangle\rangle
\nabla_i(e_{\beta}h)\Big\|_{\mathcal H^1(\mathbb R^n)}\Big[u\Big]_{\hbox{BMO}(B_{2\theta})}
\nonumber\\
&\lesssim\|\sqrt{g}g^{ij}\langle\nabla_{j}e_{\alpha},e_{\beta}\rangle\|_{L^{2}(B_{\theta})}\|\nabla(e_{\beta}h)\|_{L^{2}(B_{\theta})}\left[u\right]_{\mathrm{BMO}(B_{2\theta})}
\nonumber\\
&\lesssim\|\nabla{u}\|_{L^{2}(B_{2\theta})}\|\nabla{u}\|_{M^{p,p}(B_{1})}\cdot\theta^{\frac{n}{p}-\frac{n}{2}},
\end{align}
where we have used:\\
(i) Since $\hbox{div}_g(\langle\nabla e_\alpha,e_\beta\rangle)=0$ in $B_\theta$ and $h\in W^{1,p'}_0(B_\theta)$, we have 
$\displaystyle\sqrt{g}g^{ij}\langle\nabla_{j}e_{\alpha},e_{\beta}\rangle\rangle
\nabla_i(e_{\beta}h)\in\mathcal H^1(\mathbb R^n)$ and
$$
\Big\|\sqrt{g}g^{ij}\langle\nabla_{j}e_{\alpha},e_{\beta}\rangle
\nabla_i(e_{\beta}h)\Big\|_{\mathcal H^1(\mathbb R^n)}\le
C\left\|\sqrt{g}g^{ij}\langle\nabla_{j}e_{\alpha},e_{\beta}\rangle\right\|_{L^2(B_\theta)}\left\| \nabla_i(e_{\beta}h)\right\|_{L^2(B_\theta)}.
$$
(ii) Since $p'>n$, the Sobolev embedding implies $\displaystyle h\in C^{1-\frac{n}{p'}}(B_\theta)$ and
$$\left\|h\right\|_{L^\infty(B_\theta)}\le C\theta^{1-\frac{n}{p'}}.$$
so that
$$\|\nabla(e_{\beta}h)\|_{L^{2}(B_{\theta})}
\le \|\nabla e_\beta\|_{L^2(B_\theta)}\|h\|_{L^\infty(B_\theta)}
+\|\nabla h\|_{L^p(B_\theta)}\theta^{\frac{n}{p}-\frac{n}{2}}
\leq C\theta^{\frac{n}{p}-\frac{n}{2}},$$
(iii) By Poincar\'e inequality, it holds
$$\left [u\right]_{\hbox{BMO}(B_{2\theta})}\le C\|\nabla u\|_{M^{p,p}(B_1)}.$$

Putting the estimates of $\phi_\alpha^{(1)}$ and $\phi_\alpha^{(2)}$ together, we obtain
\begin{equation}\label{phi_grad_estimate}
    \left((\tau\theta)^{p-n}\int_{B_{\tau\theta}}|\nabla\phi_{\alpha}|^{p}dx\right)^{\frac{1}{p}}
    \leq C\left[\tau^{\alpha_0}+\tau^{1-\frac{n}{p}}\epsilon_{0}\right]\|\nabla{u}\|_{M^{p,p}(B_{1})},
\ \forall 0<\tau<1.
\end{equation}

\noindent{\it{Step 4.}} Estimation of $\psi_{\alpha}$: Since $\hbox{div}_{\widetilde g}(\psi_\alpha)=0$ on
$B_{2\theta}$,  we have
\begin{align*}
\int_{B_{2\theta}}|\psi_{\alpha}|^{2}_{\widetilde g}dv_{\widetilde g}&=\int_{B_{2\theta}}\left\langle(\psi_{\alpha}+\nabla\phi_{\alpha}),\psi_{\alpha}\right\rangle_{\widetilde g}dv_{\widetilde g}\\
&=\int_{B_{2\theta}}\langle\langle\nabla((u-u_{2\theta})\eta),e_{\alpha}\rangle,\psi_{\alpha}\rangle_{\widetilde g}dv_{\widetilde g}\\
&=-\int_{B_{2\theta}} (u-u_{2\theta})\eta\langle\nabla e_{\alpha},\psi_{\alpha}\rangle_{\widetilde g}
dv_{\widetilde g}\\
&\lesssim\left\|\sqrt{\widetilde g}\ {\widetilde g}^{ij}\nabla_i e_\alpha\psi_\alpha^j\right\|_{\mathcal H^1}\left[(u-u_{2\theta})\eta\right]_{\mathrm{BMO}}\\
&\lesssim \|\psi_{\alpha}\|_{L^{2}(B_{2\theta})}\|\nabla{e}_{\alpha}\|_{L^{2}(B_{2\theta})}
\left[(u-u_{2\theta})\eta\right]_{\mathrm{BMO}}\\
&\lesssim\|\nabla{u}\|_{L^{2}(B_{2\theta})}\|\psi_{\alpha}\|_{L^{2}(B_{2\theta})}\|\nabla{u}\|_{M^{p,p}(B_{1})},
\end{align*}
where we have used the fact
$$\left[(u-u_{2\theta})\eta\right]_{\mathrm{BMO}}
\le C\left[u\right]_{\mathrm{BMO(B_{2\theta})}}\le C\left\|\nabla u\right\|_{M^{p,p}(B_1)}.
$$
This, combined with H\"older's inequality,  implies
\begin{equation}\label{psi_estimate}
\left(\theta^{p-n}\int_{B_{\theta}}|\psi_{\alpha}|^{p}\right)^{\frac{1}{p}}
\leq \,C\epsilon_{0}\left\|\nabla{u}\right\|_{M^{p,p}(B_{1})}.
\end{equation}
\noindent{\it Step 5}. Decay estimation of $\nabla u$:
Putting (\ref{phi_grad_estimate}) and (\ref{psi_estimate}) together, we have
that for some $0<\alpha_0<1$,
\begin{equation}\label{411}
   \left((\tau\theta)^{p-n}\int_{B_{\tau\theta}}|\nabla{u}|^{p}\right)^{\frac{1}{p}}
   \leq C\Big(\epsilon_0+\tau^{\alpha_0}+\tau^{1-\frac{n}{p}}\epsilon_{0}\Big)
\Big\|\nabla{u}\Big\|_{M^{p,p}(B_{1})}
\end{equation}
holds  for any $0<\tau<1$ and $0<\theta<\frac12$.
Now we claim that for some $\alpha_0\in (0,1)$, it holds
\begin{equation}\label{thm4c}
 \left \|\nabla u\right\|_{M^{p,p}(B_{\frac{\tau}{4}})}\leq
C\Big(\epsilon_0+\tau^{\alpha_0}+\tau^{1-\frac{n}{p}}\epsilon_{0}\Big)
\left\|\nabla{u}\right\|_{M^{p,n-p}(B_{1})}, \ \forall 0<\tau<1.
\end{equation}
To show (\ref{thm4c}), let $B_{s}(y)\subset{B}_{\frac{\tau}{4}}$. We divide it into three cases:\\
(a) $y\in{B}_{\frac{\tau}{4}}\cap{B}^{\pm}$ and $s<|y^{n}|$. As remarked in the begin of proof,
we have that for some $0<\alpha_0<1$,
\begin{align*}
\left(s^{p-n}\int_{B_{s}(y)}|\nabla{u}|^{p}\right)^{\frac{1}{p}}&
\leq C\left(\frac{s}{|y^{n}|}\right)^{\alpha_0}\left(|y^{n}|^{p-n}\int_{B_{|y^{n}|}(y)}|\nabla{u}|^{p}\right)^{\frac{1}{p}}\\
&\leq C\left(\frac{s}{|y^{n}|}\right)^{\alpha_0}\left((2|y^{n}|)^{p-n}\int_{B_{2|y^{n}|}(y',0)}|\nabla{u}|^{p}\right)^{\frac{1}{p}}\\
&\leq C\left(\left(\frac{\tau}{2}\right)^{p-n}\int_{B_{\frac{\tau}{2}}(y',0)}|\nabla{u}|^{p}\right)^{\frac{1}{p}}
\ (\hbox{since} \ |y^n|\le\frac{\tau}4) \\
&\leq C(\epsilon_0+\tau^{\alpha_0}+\tau^{1-\frac{n}{p}}\epsilon_{0})\|\nabla{u}\|_{M^{p,p}(B_{1})}
\ (\hbox{by}\ (\ref{411})).
\end{align*}
\noindent (b) $y\in{B}_{\frac{\tau}{4}}\cap{B}^{\pm}$ and $s\ge |y^{n}|$. Then we have
$\displaystyle B_{s}(y)\subset{B}_{|y^{n}|+s}(y',0)\subset{B}_{2s}(y',0)$. Hence
\begin{align*}
\left(s^{p-n}\int_{B_{s}(y)}|\nabla{u}|^{p}\right)^{\frac{1}{p}}&
\leq 2^{\frac{n-p}{p}}\left((2s)^{p-n}\int_{B_{2s}(y',0)}|\nabla{u}|^{p}\right)^{\frac{1}{p}}\\
&\leq C\Big(\epsilon_0+\tau^{\alpha_0}+\tau^{1-\frac{n}{p}}\epsilon_0\Big)
\left\|\nabla{u}\right\|_{M^{p,p}(B_{1})} \ (\hbox{by}\ (\ref{411})).
\end{align*}
\noindent (c) $y\in{B}_{\frac{\tau}{4}}\cap\Gamma_1$, i.e. $y^n=0$.  Then it follows directly
from (\ref{411}) that
$$\left(s^{p-n}\int_{B_s(y)}|\nabla u|^p\right)^\frac{1}{p}
\le C\Big(\epsilon_0+\tau^{\alpha_0}+\tau^{1-\frac{n}{p}}\epsilon_0\Big)
\left\|\nabla u\right\|_{M^{p,p}(B_1)}.$$
Combining (a), (b) and (c) together and taking supremum over
all $B_s(y)\subset B_{\frac{\tau}4}$,  we obtain \eqref{thm4c}.

It is now clear that by first choosing sufficiently small $\tau$ and then
sufficiently small $\epsilon_{0}$, we have
$$\left\|\nabla u\right\|_{M^{p,p}(B_{\frac{\tau}4})}\le \frac12\left\|\nabla u\right\|_{M^{p,p}(B_1)}.$$
Iterating this inequality finitely many time yields that there exists $\alpha_1\in (0,1)$ such that  for any $x\in B_\frac14$ and $0<r\le \frac12$, it holds
$$r^{p-n}\int_{B_r(x)}\left|\nabla{u}\right|^{p}\,dx\leq C\,r^{p\alpha_1}
\left\|\nabla{u}\right\|_{M^{p,p}(B_{1})}^p.$$
This immediately implies $\displaystyle u\in{C}^{\alpha_1}(B_{\frac{1}{2}})$. The proof is now completed.
\end{proof}

\section{Lipschitz and piecewise $C^{1,\alpha}$-estimate}

In this section, we will first establish both Lipschitz and piecewise $C^{1,\alpha}$-regularity
for stationary harmonic maps on domains with piecewise $C^{0,1}$-metrics, under
a smallness condition of energy. Then we will sketch a proof of Theorem
\ref{thm1}.

\begin{theorem}\label{thm51}There exist $\epsilon_{0}>0$ and $\beta_{0}\in(0,1)$ depending only on $n, g$ such that if
the metric $g\in{C}^{0,1}(B_{1}^{\pm}\cup\Gamma_1)$ satisfies the condition (\ref{jump_cond})
on $\Gamma_1$, and
$u\in{W}^{1,2}(B_{1},N)$ is a stationary harmonic map on $(B_1,g)$ satisfying
\begin{equation}\label{small_norm_energy1}
r_0^{2-n}\int_{B_{r_0}(x_0)}|\nabla{u}|_{g}^{2}\,dv_g\leq\epsilon_{0}^{2}
\end{equation}
for some $x_0\in{B}_{\frac{1}{2}}$ and $0<r_0\leq\frac{1}{4}$, then
$u\in{C}^{1,\beta_{0}}\Big(B_{\frac{r_0}{2}}(x_0)\cap\overline{{B}^{\pm}},N\Big)$,
and $u\in C^{0,1}\Big(B_{\frac{r_0}2}(x_0), N\Big)$.
\end{theorem}

\begin{proof} The proof is based on both the hole filling argument and freezing coefficient method. It is
divided into two steps.

\noindent {\it {Step 1.}} $u\in C^\alpha(B_{\frac{3r_0}4}(x_0),N)$ for any $0<\alpha<1$.
To see this, recall Theorem \ref{thm41} implies that there exists $0<\alpha_0<\frac23$ such that
$u\in{C}^{\alpha_{0}}(B_{\frac{7r_0}{8}}(x_0))$ and for any $y\in{B}_{\frac{7r_0}{8}}(x_0)$,
it holds
\begin{equation}\label{51}
    s^{2-n}\int_{B_{s}(y)}|\nabla{u}|^{2}\,dx
    \leq\,C\left(\frac{s}{r}\right)^{2\alpha_{0}}r^{2-n}\int_{B_r(y)}|\nabla{u}|^{2}\,dx,
\ 0<s\leq r<\frac{r_0}8,
\end{equation}
and
\begin{equation}\label{52}
\mathrm{osc}_{B_r(y)}u\leq\,Cr^{\alpha_{0}}, \ 0<r<\frac{r_0}8.
\end{equation}
For $y\in B_{\frac{7r_0}8}(x_0)$ and $0<r<\frac{r_0}8$,
let $v:B_r(y)\to \mathbb R^k$ solve
\begin{equation}\label{510}
\begin{cases}
\Delta_{g}v=0&\mbox{in}~B_{r}(y)\\
v=u&\mbox{on}~\partial{B}_r(y).
\end{cases}
\end{equation}
Then by the maximum principle and (\ref{52}), we have
$$
\mathrm{osc}_{B_{r}(y)}v\leq\mathrm{osc}_{\partial{B}_{r}(y)}u\le Cr^{\alpha_0}.
$$
Moreover, since $g\in C^{0,1}(B_1^\pm\cup\Gamma_1)$, it is well-known (cf. \cite{ln} Theorem 1.1) that
$v\in C^{0,1}(B_{\frac{r}{2}}(y),\mathbb{R}^{k})$ and
$v\in{C}^{1,\beta}(B_{\frac{r}{2}}(y)\cap\overline{{B}^{\pm}},\mathbb R^k)$ for any $0<\beta<1$.

Now multiplying both the equations (\ref{harmonic}) and (\ref{510})  by $(u-v)$ and
subtracting each other and then integrating over
$B_{r}(y)$, we obtain
$$\int_{B_{r}(y)}|\nabla(u-v)|^{2}\,dx\lesssim\int_{B_{r}(y)}|\nabla{u}|^{2}|u-v|\lesssim\,r^{n-2+3\alpha_{0}}.$$
Since
$$\int_{B_{\frac{r}{2}}(y)}|\nabla{v}|^{2}\,dx\leq C\Big\|\nabla v\Big\|_{L^\infty(B_{\frac{r}2(y)})}^2r^{n},$$
we obtain
$$\left(\frac{r}{2}\right)^{2-n}\int_{B_{\frac{r}{2}}(y)}|\nabla{u}|^{2}\,dx
\leq C\Big(\left\|\nabla v\right\|_{L^\infty(B_{\frac{r}2}(y))}^2 r^2+r^{3\alpha_{0}}\Big)
\leq Cr^{3\alpha_{0}}.$$
This, combined with Morrey's decay lemma, yields
$\displaystyle u\in{C}^{\frac{3\alpha_0}2}(B_{\frac{7r_0}8}(x_0))$. Repeating this argument, we
can show that $\displaystyle u\in{C}^{\alpha}(B_{\frac{3r_0}{4}}(x_0))$ for any
$0<\alpha<1$, and
\begin{equation}\label{520}
r^{2-n}\int_{B_{r}(y)}|\nabla{u}|^{2}\,dx\leq\,Cr^{2\alpha},
\ \forall y\in{B}_{\frac{3r_0}{4}}(x_0), \ 0<r<\frac{r_0}{4}.
\end{equation}
\noindent{\it Step 2}. There exists $0<\beta_0<1$ such that
 $u\in C^{1,\beta_0}\Big(B_{\frac{r_0}{2}}(x_0)\cap \overline{B^\pm}, N\Big)$.
The proof is divided into two cases.

\noindent{\it {Case I.}} $\displaystyle x_0=(x_0',x_0^{n})\in{B}_1^{\pm}$. We may assume $\displaystyle 0<r_0<|x_0^{n}|$ so that $B_{r_0}(x_0)\subset B^\pm$.
For $B_{r}(x)\subset B_{r_0}(x_0)$,  let $v:B_r(x)\to\mathbb R^k$ solve
\begin{equation}\label{530}
\begin{cases}
\Delta_{g}v=0 &\mbox{in}~B_{r}(x),\\
v=u &\mbox{on}~\partial{B}_{r}(x).
\end{cases}
\end{equation}
Then by Step 1, we have that for any $\frac23<\alpha<1$,
\begin{equation}\label{56}
\int_{B_{r}(x)}|\nabla(u-v)|^{2}\,dx\leq C\int_{B_{r}(x)}|\nabla{u}|^{2}|u-v|\,dx\le C\,r^{3\alpha+n-2}.
\end{equation}
Moreover, since $g\in C^{0,1}(B_{r_0}(x_0))$, we have that for any $0<\beta<1$, $v\in{C}^{1,\beta}(B_{\frac{r}2}(x))$ and
\begin{equation}\label{57}
\fint_{B_{s}(x)}\left|\nabla{v}-(\nabla{v})_{B_{s}(x)}\right|^{2}\,dx\leq C
(\frac{s}{r})^{2\beta}\fint_{B_{r}(x)}\left|\nabla{u}-(\nabla{u})_{B_r(x)}\right|^{2}\,dx,\quad0<s\le\frac{r}2.
\end{equation}
Henceforth, we denote  $\displaystyle\fint_{E}f=\frac{1}{|E|}\int_E f\,dx$.
Combining \eqref{56} and \eqref{57}\footnote{note that (\ref{57}) trivially
holds for $\frac{r}2\le s\le r$.}, we obtain that  for any $0<\theta<1$,
\begin{align*}
\fint_{B_{\theta{r}}(x)}\left|\nabla{u}-(\nabla{u})_{B_{\theta{r}}(x)}\right|^{2}\,dx&
\leq 2\Big[\fint_{B_{\theta{r}}(x)}\left|\nabla{u}-\nabla{v}\right|^{2}\,dx+\fint_{B_{\theta{r}}(x)}\left|\nabla{v}-(\nabla{v})_{B_{\theta{r}}(x)}\right|^{2}\,dx
\Big]\\
&\leq C\Big[\theta^{2\beta}\fint_{B_{r}(x)}\left|\nabla{u}-(\nabla{u})_{B_{r}(x)}\right|^{2}\,dx
+\theta^{-n}r^{3\alpha-2}\Big].
\end{align*}
For $\frac{3\alpha-2}{2}<\beta_0<\beta$, let $0<\theta_0<1$ be such that $C\theta_0^{2\beta}
=\theta_0^{2\beta_0}$. Then we have
\begin{equation}\label{540}
\fint_{B_{\theta_0{r}}(x)}\left|\nabla{u}-(\nabla{u})_{B_{\theta_0{r}}(x)}\right|^{2}\,dx
\le\theta_0^{2\beta_0}
\fint_{B_{{r}}(x)}\left|\nabla{u}-(\nabla{u})_{B_{{r}}(x)}\right|^{2}\,dx
+Cr^{3\alpha-2}.
\end{equation}
Iterating (\ref{540}) $m$-times, $m\ge 1$,  yields
\begin{eqnarray}\label{550}
\fint_{B_{\theta_0^m{r}}(x)}\left|\nabla{u}-(\nabla{u})_{B_{\theta_0{r}}(x)}\right|^{2}\,dx
&\le&\left(\theta_0^m\right)^{2\beta_0}
\fint_{B_{{r}}(x)}\left|\nabla{u}-(\nabla{u})_{B_{{r}}(x)}\right|^{2}\,dx\nonumber\\
&&+C(\theta_0^mr)^{3\alpha-2}\sum_{j=1}^m \theta_0^{j(2\beta_0-(3\alpha-2))}\\
&\le& (\theta_0^m)^{3\alpha-2}\Big[\fint_{B_{{r}}(x)}\left|\nabla{u}-(\nabla{u})_{B_{{r}}(x)}\right|^{2}\,dx+   Cr^{3\alpha-2}\Big].\nonumber
\end{eqnarray}
This clearly implies that $\displaystyle\nabla u\in C^{\frac{3\alpha-2}2}(B_{r_0}(x_0))$.

\medskip
\noindent{\it{Case II.}}
$x_0=(x_0',0)\in\Gamma_1$. For simplicity, we assume $x_0'=0$.  Define the piecewise constant
metric $\bar g$ on $B_1$ by letting
$$\bar{g}(x)=\begin{cases} \lim_{t\downarrow 0^+} g(0', t) &  \ x\in B_1^+\\
 \lim_{t\uparrow 0^-} g(0', t) &  \ x\in B_1^-.
\end{cases}
$$
Then we have
\begin{equation}\label{piece_lip}
\left|g(x)-\bar{g}(x)\right|\le C|x|, \ x\in B_1.
\end{equation}
Moreover, by suitable dilations and rotations of the coordinate system,  (\ref{jump_cond}) implies that
there exists a positive constant $k\not=1$ such that
$$\bar g(x)=(1+(k-1)\chi_{B_1^-}(x)) g_0,\  x\in B_1,$$
where  $\displaystyle\chi_{B_1^-}$
is the characteristic function of $B_1^-$.

For $0<r<\frac{r_0}2$, let $v:B_r(0)\to\mathbb R^k$ solve
\begin{equation}\label{harm_fun}
\begin{cases}
\Delta_{\bar g} v=0 & \ {\rm{in}}\ B_{r}(0),\\
v=u & \ {\rm{on}}\ \partial B_r(0).
\end{cases}
\end{equation}
Then we have
$$\hbox{osc}_{B_r(0)} v\le \hbox{osc}_{B_r(0)} u\le Cr^\alpha, \
\ \int_{B_r(0)}|\nabla v|^2\,dx\le C\int_{B_r(0)}|\nabla u|^2\le Cr^{n-2+2\alpha}.$$
Multiplying (\ref{harmonic}) and (\ref{harm_fun}) by $(u-v)$ and integrating over $B_r(0)$, we obtain
\begin{eqnarray*}
&&\int_{B_r(0)}|\nabla(u-v)|^2\,dx\nonumber\\
&\le&\int_{B_r(0)} g^{ij} (u-v)_i (u-v)_j \sqrt{g}\,dx\nonumber\\
&\le & C\int_{B_r(0)}|\nabla u|^2 |u-v|\,dx
+\int_{B_r(0)}|\sqrt{g}g^{ij}-\sqrt{\bar{g}}{\bar g}^{ij}||v_i||(u-v)_j|\,dx\nonumber\\
&\le& C \hbox{osc}_{B_r(0)} v\int_{B_r(0)}|\nabla u|^2\,dx +Cr^2\int_{B_r(0)}|\nabla v|^2+\frac12\int_{B_r(0)}|\nabla(u-v)|^2\,dx
\nonumber\\
&\le& Cr^{n-2+3\alpha}+Cr^{n+\alpha}+\frac12\int_{B_r(0)}|\nabla(u-v)|^2\,dx.
\end{eqnarray*}
This implies
\begin{equation}
\int_{B_r(0)}|\nabla(u-v)|^2\,dx\le Cr^{n-2+3\alpha}.\label{500}
\end{equation}
It is well-known that $v\in C^\infty\Big(\overline{B_s^\pm(0)}\Big)$ for any $0<s<r$. In fact,
(\ref{harm_fun}) is equivalent to:
\begin{equation}\label{harm_fun1}
\frac{\partial}{\partial x_i}\Big((1+(k^{\frac{n}2}-1)\chi_{B_1^-})\frac{\partial v}{\partial x_i}\Big)=0,
\ {\rm{in}}\ B_r(0),
\end{equation}
we conclude \\
(i) $\frac{\partial v}{\partial x_n}$ satisfies the jump property on $\Gamma_1$:
$$
\lim_{x_n\downarrow 0^+} \frac{\partial v}{\partial x_n}(x', x_n)
=k^{\frac{n}2}\lim_{x_n\uparrow 0^-} \frac{\partial v}{\partial x_n}(x', x_n),
\ \ \forall (x',0)\in \Gamma_1\cap B_r(0).
$$
(ii) $\nabla^\alpha v\in C^0(B_r(0))$ for any multi-index $\alpha=(\alpha_1,\cdots,
\alpha_{n-1}, 0)$. \\
(iii) $\nabla v\in L^\infty(B_s(0))$ for any $0<s<r$, and
\begin{equation}\label{lip_bound}
\left\|\nabla v\right\|^2_{L^\infty(B_{\frac{r}2}(0))}\le C r^{2-n}\int_{B_r(0)}|\nabla u|^2.
\end{equation}
For $f:B_r(0)\to\mathbb R^k$, set
\begin{equation}\label{new_deriv}
\widetilde {D}f:=\Big(\frac{\partial f}{\partial x_1},
\cdots, \frac{\partial f}{\partial x_{n-1}}, (1+(k^{\frac{n}2}-1)\chi_{B_1^-})\frac{\partial f}{\partial x_n}
\Big),
\end{equation}
and let $\displaystyle\Big({\widetilde{D}f}\Big)_s=\fint_{B_s(0)}\widetilde{D} f\,dx$
denote the average of $\widetilde {D}f$ over $B_s(0)$.
Then we have that  for any $0<\beta<1$,
\begin{equation}\label{580}
\fint_{B_s(0)}\left|\widetilde{D}v-(\widetilde{D}v)_s\right|^2\,dx
\le C\left(\frac{s}{r}\right)^{2\beta}
\fint_{B_r(0)}\left|\widetilde{D}u-(\widetilde{D}u)_r\right|^2\,dx,
\ \forall 0<s\le r.
\end{equation}
Combining (\ref{500}) with (\ref{580}) yields that for any $0<\theta<1$,
\begin{equation}\label{591}
\fint_{B_{\theta r}(0)}\left|\widetilde{D}u-(\widetilde{D}u)_{\theta r}\right|^{2}\,dx\le
C\theta^{2\beta}\fint_{B_{r}(0)}\left|\widetilde{D}u-(\widetilde{D}u)_r\right|^{2}\,dx
+C\theta^{-n}r^{3\alpha-2}.
\end{equation}
As in Case I, iteration of (\ref{591}) yields that for any $0<s\le r$, it holds
\begin{equation}\label{592}
\fint_{B_s(0)}\left|\widetilde{D}u-(\widetilde{D}u)_s\right|^{2}\,dx\le
C\left(\frac{s}{r}\right)^{3\alpha-2}\fint_{B_{r}(0)}\left|\widetilde{D}u-(\widetilde{D}u)_r\right|^{2}\,dx
+Cs^{3\alpha-2}.
\end{equation}
This, combined with Case I, can imply that for any $B_r(x)\subset B_{r_0}(x_0)$ and
$0<s\le r$,
\begin{equation}\label{593}
\fint_{B_s(x)}\left|\widetilde{D}u-(\widetilde{D}u)_{x,s}\right|^{2}\,dx\le
C\left(\frac{s}{r}\right)^{3\alpha-2}\fint_{B_{r}(x)}\left|\widetilde{D}u-(\widetilde{D}u)_{x,r}\right|^{2}\,dx
+Cs^{3\alpha-2},
\end{equation}
where $(\widetilde{D}u)_{x,s}$ denotes the average of $\widetilde{D}u$ over $B_s(x)$.
It is readily seen that (\ref{593}) yields $\displaystyle
\nabla u\in C^{1,\frac{3\alpha-2}2}(B_{\frac{r_0}2}(x_0)
\cap \overline{B_1^\pm})$ and $\displaystyle
u\in C^{0,1}(B_{\frac{r_0}2}(x_0))$. This completes the proof.
\end{proof}

Now we sketch the proof of Theorem \ref{thm1}.

\begin{proof}[Proof of Theorem \ref{thm1} ]

Define the singular set
$$\Sigma=\left\{x\in\Omega:~\varliminf\limits_{r\rightarrow0}r^{2-n}\int_{B_{r}(x)}|\nabla{u}|^{2}\,dx\geq\epsilon^{2}_{0}\right\}.$$
Then by a covering argument we have $\displaystyle H^{n-2}(\Sigma)=0$ (see, for example,
Evans-Gariepy \cite{EG}). For
any $x_{0}\in\Omega\setminus\Sigma$, there exists $\displaystyle
0<r_{0}<\hbox{dist}(x_0,\partial\Omega)$ such that
$$r_{0}^{2-n}\int_{B_{r_{0}}(x)}|\nabla{u}|^{2}\,dx\leq\epsilon^{2}_{0}.$$
Hence by Theorem \ref{thm21}, Theorem \ref{thm41}, and Theorem \ref{thm51},
we have
$$\displaystyle u\in C^{1,\alpha}\Big(B_{\frac{r_0}2}(x_0)\cap\overline{\Omega^\pm},N\Big)
\ {\rm{and}}\ \displaystyle u\in C^{0,1}\Big(B_{\frac{r_0}2}(x_0), N\Big),$$
for some $0<\alpha<1$. In particular, we have
$$\lim_{r\downarrow 0}r^{2-n}\int_{B_r(x)}|\nabla u|^2\,dx=0, \ \forall x\in B_{\frac{r_0}2}(x_0),$$
so that $B_{\frac{r_0}2}(x_0)\cap\Sigma=\emptyset$ and hence $\Sigma$ is closed.
This completes the proof.
\end{proof}

\section{Harmonic maps to manifolds supporting convex distance functions}

In this section, we consider weakly harmonic maps $u$ from $(\Omega, g)$, with $g$ the piecewise Lipschitz continuous metric as in Theorem 1.1, to $(N,h)$, whose universal cover $(\widetilde N,
\widetilde h)$ supports a convex distance function square $d^2_{\widetilde N}(\cdot,p)$ for
any $p\in\widetilde{N}$. We will establish both the global Lipschitz continuity and
piecewise $C^{1,\alpha}$-regularity  for such harmonic maps $u$. This can be viewed
as a generalization of the well-known regularity theorem by
Eells-Sampson \cite{ES} and  Hildebrandt-Kaul-Widman \cite{HKW}.

The crucial step is the following theorem on H\"older continuity.
\begin{theorem}\label{thm61}
Assume that the metric $g$ is bounded measurable on $\Omega$, i.e. there exist two constants $0<\lambda<\Lambda<+\infty$ such that $\displaystyle \lambda \mathbb I_n\le g(x)
\le \Lambda\mathbb I_n$ for a.e. $x\in\Omega$. Assume also that the universal cover
$(\widetilde N,
\widetilde h)$ of $(N, h)$ supports a convex distance function square $d^2_{\widetilde N}(\cdot,p)$ for
any $p\in\widetilde{N}$. If
$\displaystyle u\in H^1(\Omega, N)$ is a weakly harmonic map, then there exists $\alpha\in (0,1)$
such that $\displaystyle u\in {C}^{\alpha}(\Omega,N)$.
\end{theorem}

\begin{proof} Here we sketch a proof that is based on modifications of that by Lin \cite{l}.
Similar ideas have been used by Evans in his celebrated work \cite{evans1} and Caffarelli \cite{c}
for quasilinear systems under smallness conditions.
First, by lifting $u:\Omega\to N$ to a harmonic map $\widetilde{u}:\Omega\to \widetilde N$, we may
simply assume $(N,h)=(\widetilde N, \widetilde h)$ and $d^2_N(\cdot, p)$ is convex on $N$ for
any $p\in N$.

We first claim that 
\begin{equation}\label{subharm}
\Delta_{g}d^{2}(u,p)\geq\ 0.
\end{equation}
In fact, by the chain rule of harmonic maps (cf.  Jost \cite{j}), we have
$$\Delta_{g}d^{2}(u,p)=\nabla_{u}d^{2}(u,p)(\Delta_{g}u)+\nabla^{2}_{u}d^{2}(u,p)(\nabla{u},\nabla{u})_{g}.$$
Since $\Delta_{g}u\perp{T}_{u}N$, $\nabla_{u}d^{2}(u,p)\in{T}_{u}N$,
the first term in the right hand side vanishes. By the convexity of $d_N^2$,
the second term in the right hand side satisfies
$$\nabla^{2}_{u}d^{2}(u,p)(\nabla{u},\nabla{u})_{g}\geq 0.$$
Since $u\in H^1(\Omega, N)$,
by suitably choosing $p\in N$ and applying Poincar\'e inequality and
Harnack's inequality, (\ref{subharm}) implies
$\displaystyle u\in L^\infty_{\hbox{loc}}(\Omega, N)$.

For a set $E\subset N$, let  ${\rm{diam}}_N(E)$ denote the diameter of $E$ with respect to the distance
function $d_N(\cdot,\cdot)$. For any ball $B_r(x)\subset \Omega$, we want to show that
$u\in C^\alpha(B_{\frac{r}2}(x))$ for some $0<\alpha<1$. To do it, denote
$$\mathcal {C}_r:=\mathrm{diam}_N\left(u(B_{r}(x))\right)<+\infty.$$
We may assume $\mathcal{C}_r>0$ (otherwise, $u$ is constant on $B_r(x)$ and we are done).
Now we want to show that there exists $0<\delta_{0}=\delta_{0}(N)\le\frac12$
 such that
\begin{equation}\label{diam}\displaystyle\mathrm{diam}_N\left(u(B_{\delta_0r}(x))\right)
\leq \frac12 \mathcal{C}_r.
\end{equation}
Since  $u_r(y)=u(x+ry):B_1(0)\to N$ is a harmonic map  $(B_1(0), g_r)$, with $g_r(y)=g(x+ry)$,
we may, for simplicity, assume $x=0$ and $r=2$.
For any $0<\epsilon<\frac12$, since $u(B_1)\subset N$ is a bounded set,
there exists $m=m(\epsilon)\ge 1$ such that $u(B_1)$ is covered by
$m$ balls $B^1,\cdots, B^m$ of radius $\epsilon\mathcal{C}_1$.  Now we have\\
\noindent{\bf Claim}:
{\it  There exists a sufficiently small
$\epsilon>0$ such that $u(B_{\frac12})$ can be covered by at most $(m-1)$ balls among $B^1,\cdots, B^m$}.

To see this, let $x_i\in B_1$ such that $B^i\subset B_{2\epsilon\mathcal C_1}(p_i), \ p_i=u(x_i)$, for $1\le i\le m$.  Let
$1\le m'\le m$ be the maximum number of points in $\{p_i\}_{i=1}^m$ such that  the distance between any two of them is at
least $\frac{1}{32}\mathcal C_1$. Thus $B_{\frac{1}{16}\mathcal C_1}(p_i), 1\le i\le m'$, covers
$u(B_1)$.
Then there exists  $i_0\in \{1,\cdots, m'\}$ such that
\begin{equation}\label{triangle}
\frac{1}{4}\mathcal{C}_1^2\leq\sup_{x\in{B}_{2}}d^{2}_N(u(x),p_{i_0})\leq \mathcal{C}_1^2,
\end{equation}
and
\begin{equation}\label{size}
H^{n}\left(u^{-1}(B^{N}(p_{i_0}, \frac{1}{16}\mathcal{C}_1))\cap{B}_{1}\right)\geq\,c_{0},
\end{equation}
for some universal constant $c_0>0$,
where $B^N(p_{i_0}, R)$ is the ball in $N$ with center $p_{i_0}$ and radius $R$.

In fact, since
$$B_{1}\subset\bigcup_{i=1}^{m'}u^{-1}\left(B^N(p_{i},\frac{1}{16}\mathcal{C}_1)\right),$$
we have
$$\sum_{i=1}^{m'}H^{n}\left(u^{-1}(B^N(p_{i},\frac{1}{16}\mathcal{C}_1))\cap B_1\right)\geq
H^n(B_{1}).$$
Hence there exists $i_{0}\in\{1,\cdots,m'\}$ such that
$$H^n\left(u^{-1}(B^N(p_{i_{0}},\frac{1}{16}\mathcal{C}_1))\right)\geq c_0:=\frac1{m'}H^n(B_{1}).$$
This implies (\ref{size}). By the triangle inequality, (\ref{triangle}) also
holds.

Define
$$f(x):=\sup\limits_{z\in{B}_{1}}d^{2}_N(u(z),p_{i_0})-d^{2}_N(u(x),p_{i_0}),\ x\in B_1.$$
It is clear that $f\geq0$ in $B_{1}$, and (\ref{subharm}) implies
$$\Delta_{g}f\leq0, \ {\rm{in}}\ B_1.$$
By Moser's Harnack inequality, we have
\begin{align}\label{less}
\inf_{B_{\frac12}}f&\ge C\fint_{B_{1}}f\ge C\int_{B_{\frac12}}f\ge C\int_{B_{\frac12}\cap{u}^{-1}(B^N(p_{i_0},\frac{1}{16}
\mathcal{C}_1))}f\nonumber\\
&\ge C\left(\sup\limits_{{B}_{1}}d^{2}_N(u,p_{i_0})-\sup\limits_{B_{1}\cap{u}^{-1}(B^N(p_{i_0},\frac{1}{16}\mathcal{C}_1))}d^{2}_N(u,p_{i_0})\right)
H^{n}\left(B_{\frac12}\cap{u}^{-1}(B^N(p_{i_0},\frac{1}{16}\mathcal{C}_1))\right)\nonumber\\
&\ge C\left(\frac{1}{4}\mathcal {C}_1^2-\frac{1}{256}\mathcal{C}_1^2\right)c_{0}:=\theta_0^2\mathcal C_1^2
\end{align}
for some universal constant $\theta_0>0$. This implies
\begin{equation}\label{less1}
\sup_{z\in B_1}d_N(u(z),p_{i_0})-\sup_{z\in B_\frac12}d_N(u(z),p_{i_0})
\ge \theta_0\mathcal C_1=(1-\theta_0)\mathcal C_1.
\end{equation}

Now we argue that the claim follows from (\ref{less1}).  For, otherwise, we would have
that $\displaystyle u(B_\frac12)\cap B_{2\epsilon\mathcal{C}_1}(p_j)\not=\emptyset$
for all $1\le j\le m$. Let $z_0\in B_1$ be such that
$$\epsilon\mathcal{C}_1+d_N(u(z_0),p_{i_0})\ge \sup_{B_1}d_N(u(z),p_{i_0}).$$
Since $\displaystyle u(B_1)\subset\cup_{i=}^m B_{2\epsilon\mathcal{C}_1}(p_i)$,  there exists $p_{i_1}\in \{p_1,\cdots, p_m\}$ such that
$u(z_0)\in B_{2\epsilon\mathcal{C}_1}(p_{i_1})$. Since $\displaystyle u(B_\frac12)\cap  B_{2\epsilon\mathcal{C}_1}(p_{i_1})\not=\emptyset$,
there exists $z_1\in B_\frac12$ such that $u(z_1)\in B_{2\epsilon\mathcal{C}_1}(p_{i_1})$.  Therefore we have
$d_N(u(z_1),u(z_0))\le 2\epsilon\mathcal{C}_1$. Therefore we have
\begin{eqnarray*}\sup_{z\in B_1}d_N(u(z),p_{i_0})-\sup_{z\in B_\frac12}d_N(u(z),p_{i_0})
&\le& \epsilon\mathcal{C}_1+d_N(u(z_0),p_{i_0})-d_N(u(z_1),p_{i_0})\\
&\le& \epsilon\mathcal{C}_1+d_N(u(z_0), u(z_1))\le 3\epsilon\mathcal{C}_1,
\end{eqnarray*}
this contradicts (\ref{less1}) if $\epsilon>0$ is chosen to be sufficiently small.

From this claim, we have either \\
(i) $\displaystyle {\rm{diam}}_N(u(B_\frac12))\le\frac12 \mathcal {C}_1$. Then (\ref{diam}) holds  with $\delta_0= \frac12$, or\\
(ii)  $\displaystyle {\rm{diam}}_N(u(B_\frac12))>\frac12 \mathcal {C}_1$.  Then we consider $v(x)=u(\frac12x):B_1\to N$ and have
\begin{itemize}
\item $v$ is a harmonic map on $\displaystyle (B_1,g_{\frac12})$,  with the metric $
\displaystyle g_{\frac12}(x)=g(\frac12x)$.
\item  $\displaystyle\frac12\mathcal{C}_1<{\rm{diam}}_N(v(B_1))
\le \mathcal{C}_1$.
\item $v(B_1)$ is covered by at most $(m-1)$ balls $B_1,\cdots, B^{m-1}$ of radius $\epsilon\mathcal{C}_1$.
\end{itemize}
Thus the claim is applicable to $v$ so that  $u(B_{\frac14})=v(B_{\frac12})$ can be covered by at most $(m-2)$ balls among $B^1,\cdots, B^{m-1}$.

If $\displaystyle {\rm{diam}}_N(v(B_{\frac12}))\le \frac12\mathcal{C}_1$,  we are done.
Otherwise, we can repeat the above argument.
It is clear that the process can at most be repeated $m$-times, and the process will not be stopped at step  $k_0\le m$ unless $\displaystyle{\rm{diam}}_N u(B_{2^{-k_0}})\le \frac12 \mathcal{C}_1$. Thus (\ref{diam}) is proven.

It is readily seen that iteration of (\ref{diam})  implies  H\"older
continuity.
\end{proof}

\begin{proof}[{Proof of Theorem \ref{thm2}}]
First, by Theorem 6.1, and the argument from \S4, we can show that for some $0<\alpha<1$,
$$\int_{B_r(x)}|\nabla u|^2\,dx\le Cr^{n-2+2\alpha}, \ \forall B_r(x)\subset\Omega.$$
Then we can follow the same proof of Theorem 5.1 to show that $u\in C^{0,1}(\Omega)$ and
$u\in C^{1,\alpha}(\Omega^\pm\cup\Gamma,N)$.
\end{proof}

\end{document}